\documentclass[reqno,english]{amsart}

\usepackage{etex}
\usepackage{amsmath,amssymb,amsthm,bbm,mathtools,comment}
\usepackage[shortlabels]{enumitem}
\usepackage[pdftex,colorlinks,backref=page,citecolor=blue]{hyperref}
\usepackage[mathscr]{euscript}
\usepackage{tikz,babel,adjustbox}
\usepackage{cmtiup}
\usepackage{amsfonts}
\usepackage{graphicx}
\usepackage{caption}
\usepackage{verbatim}
\usepackage{array}
\usepackage[frame,cmtip,arrow,matrix,line,graph,curve]{xy}
\usepackage{graphpap, color, pstricks}
\usepackage{pifont}
\usepackage{cmtiup}
\usepackage{amssymb}
\usepackage{amsthm}
\allowdisplaybreaks

\theoremstyle{plain}
\newtheorem{theorem}{Theorem}[section]

\newtheorem{lem}[theorem]{Lemma}

\newtheorem{prop}[theorem]{Proposition}

\theoremstyle{remark}

\newcommand{\beq}[1]{\begin{equation}\label{#1}}
\newcommand{\enq}[0]{\end{equation}}

\usepackage{etex}
\usepackage{amsmath,amssymb,amsthm,bbm,mathtools,comment}
\usepackage[shortlabels]{enumitem}
\usepackage[pdftex,colorlinks,backref=page,citecolor=blue]{hyperref}
\usepackage[mathscr]{euscript}
\usepackage{tikz,babel,adjustbox}
\usepackage{cmtiup}
\usepackage{amsfonts}
\usepackage{graphicx}
\usepackage{caption}
\usepackage{verbatim}
\usepackage{array}
\usepackage[frame,cmtip,arrow,matrix,line,graph,curve]{xy}
\usepackage{graphpap, color, pstricks}
\usepackage{pifont}
\usepackage{cmtiup}
\usepackage{amssymb}

\allowdisplaybreaks

\theoremstyle{plain}
\theoremstyle{remark}

\title{ ON GENERATING SETS OF LEFT-COMPRESSED INTERSECTING FAMILIES}
\author{Nguyen Trong Tuan$^{1,2}$}
\email{23n21103@student.hcmus.edu.vn (Nguyen Trong Tuan)}
\author{Nguyen Anh Thi$^{1,2}$}
\email{nathi@hcmus.edu.vn (Nguyen Anh Thi)}
\address{$^1$Faculty of Mathematics and Computer Science, University of Science, Ho Chi Minh City, Vietnam}
 \address{$^2$ Vietnam National University, Ho Chi Minh City, Vietnam}
\author{ Tran Dan Thu$^3$}
\email{thu.tran@umt.edu.vn (Tran Dan Thu)}
\address{$^3$University of Management and Technology, Ho Chi Minh City, Vietnam}
\begin{document}
\begin{abstract} We construct a left-compressed family $\mathcal{F}(n,k,\mathcal{G})$ generated by a collection of generating sets $\mathcal{G}$, and identify conditions on $\mathcal{G}$ under which $\mathcal{F}(n,k,\mathcal{G})$ is an intersecting family. From this, we obtain a convenient method to construct left-compressed intersecting families as well as extending a left-compressed intersecting family. After that, we provide some comparisons between the theorem of Bond and our result.
\end{abstract}
\maketitle

\section{Introduction}
\subsection{Overview of the left-compressed intersecting families}
Throughout this paper, a subset $A$ of $[n]$ consisting of $r$ elements is called an $r$-set and is always listed in ascending order: $A=\{a_{1},a_{2},\ldots,a_{r}\}$ where $1\leq a_{1}< a_{2}<\ldots<a_{r} \leq n$. For each positive integer $n$, we denote $[n]=\{1,2,\ldots,n\}$. For $m \leq n$ be positive integers, we denote $[m,n]= \{m,m+1,\ldots,n\}$. \\ \\
Assume that $n$ and $k$ are positive integers satisfying $4 \leq 2k \leq n$. We denote $\binom{[n]}{k}$ as the collection of all $k$-set of $[n]$. A family $\mathcal{A}$ is {\bf \emph{$k$-uniform}} if every member of $\mathcal{A}$ contains exactly $k$ elements. A family $\mathcal{A} \subseteq 2^{[n]}$ is called {\bf \emph{ intersecting }} if for all $S,T \in \mathcal{A}$, we have $S \cap T \neq \emptyset$. An intersecting family $\mathcal{A}$ is {\bf \emph{ trivial }} if $\bigcap_{A \in \mathcal{A}}A=\emptyset$. The families $\mathcal{A}$ and $ \mathcal{B}$ are said to be {\bf \emph{cross- intersecting }} if any $A \in \mathcal{A}$ and $B \in \mathcal{B}$, we have $A \cap B \neq \emptyset$. Given two sets $A,B \in \binom{[n]}{k}$ with $A=\{a_{1},a_{2},\ldots,a_{k} \}$ and $B=\{b_{1},b_{2},\ldots,b_{k}\}$, we define $A \leq B$ if $a_{i} \leq b_{i}$ for all $i=1,2,\ldots,k$. We also define $A < B$ if $A \leq B$ and there exists $1 \leq i \leq k$ such that $a_{i} < b_{i}$. We say that two sets $A$ and $B$ (not necessarily distinct) are {\bf \emph{ strongly intersecting }} and write $A \sim_{si} B$ if for any $A',B'$ such that $A'\leq A, B'\leq B$ then $A' \cap B' \neq \emptyset$. A family is said to be {\bf\emph{ strongly intersecting}} if every pair of its members is strongly intersecting.\\
Next, we shall consider a family possessing a special property that plays an important role in our study. {The standard compression technique (also known as the shifting technique) is a useful tool for proving results in extremal set theory. In particular, it is used in the proofs of the Erdős–Ko–Rado and Hilton–Milner theorems. A left-compressed family can be defined in several ways. The first definition given here is closely related to the compression (or shifting) technique (see \cite{Fr87} for a survey). A family $\mathcal{A}\subseteq \binom{[n]}{k}$ is said to be {\bf\emph{left-compressed}} (or {\bf\emph{shifted}}) if for all $i<j$ and all $A\in \mathcal{A}$ with $j\in A$ and $i\neq A,$ the set $A'=A-\{j\}\cup \{i\}$ is also in $\mathcal{A}$.  Any family set can be transformed into a left-compressed family  by using the following operations: Define an $(i,j)$-shift of $\mathcal{A}$ to be the set system $\mathcal{A'}$ obtained by replacing $A\in \mathcal{A}$ with $j\in A$ and $i\neq A$ with $A'=A-\{j\}\cup\{i\}$ if $A'\notin \mathcal{A}.$ We can verify that upon applying an $(i,j)$-shift, an intersecting family remains intersecting. Furthermore, upon finitely many $(i,j)$-shift, we obtain a shifted (left-compressed) family. The second definition of a left-compressed family is as follows. Given two sets $A,B \in \binom{[n]}{k}$ with $A=\{a_{1},a_{2},\ldots,a_{k} \}$ and $B=\{b_{1},b_{2},\ldots,b_{k}\}$, we define $A \leq B$ if $a_{i} \leq b_{i}$ for all $i=1,2,\ldots,k$. We also define $A < B$ if $A \leq B$ and there exists $1 \leq i \leq k$ such that $a_{i} < b_{i}$. 
A family $\mathcal{A} \subseteq \binom{[n]}{k}$ is said to be {\bf\emph{left-compressed}} if for any $A \in \mathcal{A}$ and for any $ B$ such that $B \leq A$, we also have $B \in \mathcal{A}$. The equivalence of the two definitions above can be proved easily. In this paper, we primarily use the second definition to evaluate the generating sets of a left-compressed intersecting family.} For $A \in \binom{[n]}{k}$, we define $\mathcal{L}(A)=\{S \in \binom{[n]}{k}: S \leq A\}$. It is obvious that $\mathcal{L}(A)$ is left-compressed and we say that $\mathcal{L}(A)$ is {\bf \emph{ left-compression
closure }}  of $A$. A $k$-set $A$ in the left-compressed family $\mathcal{A} \subseteq \binom{[n]}{k}$ is {\bf \emph{ maximal set}} if there does not exist any $B \in \mathcal{A}$ such that $A < B$. To describe left-compressed families conveniently, we extend the order relation $"\leq"$ as follows.
For two sets $A=\{a_{1},a_{2},\ldots,a_{r}\} \in 2^{[n]}$ and $B=\{b_{1},b_{2},\ldots,b_{s} \} \in 2^{[n]}$, we define $A \preceq B$ if  $r \geq s$ and $a_{i} \leq b_{i}$ for all $1 \leq i \leq s$. Note that the order relations "$\leq$" on $\binom{[n]}{k}$ and "$\preceq$" on $2^{[n]}$ are both partial orders. Therefore, every nonempty subset of $\binom{[n]}{k}$ (or $2^{[n]}$) has at least one maximal element.\\ \\
We refer to a family that is both intersecting and left-compressed as a {\bf \emph{ left-compression intersecting family}} and we can abbreviate as {\bf \emph{ LCIF}}. It is obvious that if $\mathcal{A}$ is LCIF and $A,B \in \mathcal{A}$ then $A \sim_{si}B$, equivalently, $\mathcal{F}({A})$ and $\mathcal{F}(B)$ are cross-intersecting. Finally, we say an intersecting family $\mathcal{A} \subseteq \binom{[n]}{k}$ is {\bf \emph{ maximal}} if no other set can be added to $\mathcal{A}$ while preserving the intersecting property. Note that a maximal left- compressed intersecting family (abbreviated {\bf \emph{ MLCIF}}) is precisely a maximal intersecting family which is left-compressed. There are some well-known MLCIF families: the {\bf \emph{Star family}} $S=\{ A\in \binom{[n]}{k}: 1\in A\}$; the family $\mathcal{A}_{2,3}=\{A \in \binom{[n]}{k}:|A\cap \{1,2,3\}|\geq 2\}$ and the {\bf \emph{ Hilton-Milner family}} $\mathcal{H}\mathcal{M}=\{A\in \mathcal{S}:A\cap \{2,\ldots,k+1\} \neq \emptyset \}\cup [2,k+1]$.
\\ \\
 To study the intersection property of a left-compressed family, we first need to find a way to describe left-compressed families. We find that if $A \preceq G$ and $B \leq A$ then $B \preceq G$. This observation provides the basis for defining a left-compressed family over $\binom{[n]}{k}$. This definition is a modification of the definition of Ahlswede and Khachatrian \cite{AK}. 
Let $\mathcal{G}$ be a collection of subsets of $[n]$. We define a $k$-uniform family as follows:
    $$\mathcal{F}(n,k,\mathcal{G})=\bigg \{S \in \binom{[n]}{k} : S \preceq G \text{ for some } G \in \mathcal{G}  \bigg \}$$
$\mathcal{G}$ is called {\bf \emph{ a collection of generating sets}} and each $G \in \mathcal{G}$ is called a {\bf \emph{generator}} of $\mathcal{F}(n,k,\mathcal{G})$ . It is clear that if there exists $G \in \mathcal{G}$ such that $|G| > k$ then there does not exist $k$-set such that $S \preceq G$. Therefore, we assume that every $G \in \mathcal{G}$ contains at most $k$ elements. When there is no risk of confusion, we may write $\mathcal{F}(\mathcal{G})$ instead of $\mathcal{F}(n,k,\mathcal{G})$. In the special case when $\mathcal{G}=\{G\}$ with $G=\{g_{1},\ldots,g_{r}\}$, we can write $\mathcal{F}(G)=\mathcal{F}(g_{1},\ldots,g_{r})$.
Hence, $\mathcal{F}(G)$ consists of all $k$-sets $S$ such that $S \preceq G$. If $|G|=k$ then $\mathcal{F}(G)$ is the set of all $k$-set $S$ satisfying $S \leq G$ and in this case, we have $\mathcal{L}(G)=\mathcal{F}(G)$. We see that $\mathcal{F}(\mathcal{G})$ is left-compressed and $\mathcal{F}(\mathcal{G})=\bigcup_{G \in \mathcal{G}}\mathcal{F}(G)$. It is clear that not every $\mathcal{F}(\mathcal{G})$ is an intersecting family. For example, the families $\mathcal{F}(\{\{2,3\}\}),\mathcal{F}(\{\{1,3\},\{1,4,5\},\{2,3,5\} \}$ are intersecting, while the families $\mathcal{F}(\{\{2,4\}\}), \mathcal{F}(\{\{2,3\},\{2,4,5\} \}$ are not intersecting. Therefore, a natural question arises: under what conditions on $\mathcal{G}$ the $\mathcal{F}(\mathcal{G})$ is an intersecting family?\\
Since $\mathcal{F}(\mathcal{G})=\bigcup_{G \in \mathcal{G}}\mathcal{F}(G)$, we have $\mathcal{F}(\mathcal{G})$ is intersecting if, for every $G \in \mathcal{G}$, the family $\mathcal{F}(G)$ is intersecting and moreover, $\mathcal{F}(G)$ and $\mathcal{F}(H)$ are cross-intersecting for all pairs $G, H \in \mathcal{G}$. Hence, we conclude that $G \sim_{si} H $ for all $G,H \in \mathcal{G}$. In other words, $\mathcal{G}$ is strongly intersecting. Thus, an important problem is how to determine whether $G \sim_{si} H$. By definition, to determine whether $G \sim_{si} H$, we have to consider the intersection $G' \cap H'$ for every $G' \leq G$ and $H' \leq H$. This is a complicated task, especially when $G$ and $H$ have a complex structure. Bond \cite{BB} established a condition for two $k$-set to be strongly intersecting: {\bf \emph{Two $k$-set $A$ and $B$ are strongly intersecting if and only if there exists a pair of indices $1 \leq i,j \leq k$ such that $i+j \geq max\{a_{i},b_{j}\}$}}.\\ \\
We develop Bond’s idea to derive a condition for two arbitrary sets to be strongly intersecting. From this, we have established the condition on $\mathcal{G}$  for the family $\mathcal{F}(\mathcal{G})$ to be intersecting. The following is our main theorem. 
\begin{theorem}\label{thr} Let $\mathcal{G} \subseteq  2^{[n]}$. The family $\mathcal{F}(\mathcal{G})$ is intersecting if and only if and only if for every $G, H \in \mathcal{G}$ (possibly identical), there exists an $1 \le l \le l$ such that $\mu_{G}(l) + \mu_{H}(l) > l$. 
\end{theorem}
Now, let $\mathcal{F}(G)$ be an LCIF. Take any $G=\{g_{1},\ldots,g_{s}\} \in \mathcal{G}$. Assume that $g_{1} \geq k+1$. Consider the two sets $A=\{1,\ldots,k\}$ and $B=G \cup C$, where $C \subseteq [g_{s}+1,n]$ and $|C|=k-s$. We have $B \preceq G$, so $B \in \mathcal{F}(\mathcal{G})$. It is obvious that $A \leq B$. Hence $A \in \mathcal{F}(\mathcal{G})$. However, this is a contradiction since $A \cap B=\emptyset$. Therefore $g_{1} < k+1$. Thus, there exists a largest index $r$ such that $g_{r} < k+r$. We call this number $r$ {\bf \emph{ type}} of $G$ and define $\pi(G)=\{g_{1},\ldots,g_{r}\}$. We also define $\pi(\mathcal{G})=\{\pi(G): G \in \mathcal{G}\}$.
If $\mathcal{F}(\mathcal{G})$ is an LCIF, then it can be contained in a larger left-compressed intersecting family, as shown in the following theorem.
\begin{theorem} Suppose $\mathcal{G} \subseteq 2^{[n]}$ and $\mathcal{F}(G)$ is LCIF. Then $\mathcal{F}(\pi(\mathcal{G}))$ is also LCIF, and $\mathcal{F}(G) \subseteq \mathcal{F}(\pi(\mathcal{G})) $
\end{theorem}

The paper is organized into four parts. In the first part, we recall the basic notions and fundamental properties, and briefly describe the overall structure of the paper. The second part contains the proofs of Theorems 1.1 and 1.2, together with several related results and a comparison with Bond’s work. The third part is devoted to the proof of Theorem 1.3 and Theorem 1.4, concerning maximal left-compressed intersecting families (MLCIFs). Finally, the last part presents several open problems and possible directions for further research.

\section{Left-compressed intersecting families}
\subsection{Preliminaries}
We proceed to establish several supporting properties that will be used in the proof of Theorem 1.1. We need a function that describes the relationship between a set $X$ and an element $l$. For a non-empty set $X$ and a non-zero number $l$, we define $\mu_{X}(l)$ as the number of elements in $X$ that are less than or equal to $l$. This is a particularly useful tool when considering the strongly intersecting property of two sets. We note that if $x_{i}= l$ then $\mu_{X}(l)=i$. So we can write $x_{\mu_{X}(l)}=l$. We also note that for any $X \subseteq [n]$, we have $0 \leq \mu_{X}(l) \leq |X|,$ and for fixed $X$, $\mu_{X}(l)$ is a non-decreasing function in $l$. For $l \geq 2$, we also have $\mu_{X}(l) - \mu_{X}(l-1) \in \{0,1\}$. Next, we establish some additional properties of the function $\mu_{X}(l)$ that relate to the two sets $A,B \in 2^{[n]}$.

\begin{prop} Let $X,A,B \subseteq [n]$ be non-empty subsets. Then, for all $l \in [n]$, the following hold: \\
(i) If $A \preceq B$ then $\mu_{A} (l) \geq \mu_{B}(l)$ ;\\
(ii) If $A \leq B$ then $\mu_{A}(l) \geq \mu_{B}(l)$;\\
(iii) If $A \subseteq B $ then $\mu_{B}(l) \geq \mu_{A}(l)$.
\end{prop}
\begin{proof}
(i) Assume that $A=\{a_{1},a_{2},\ldots,a_{r}\}$ and $B=\{b_{1},b_{2},\ldots,b_{
s}\}$. Since $A \preceq B$, we have $r \geq s$ and\\ $\{a_{1},a_{2},\ldots,a_{s}\} \leq \{b_{1},b_{2},\ldots,b_{s}\}$
Now, we consider different cases for $l \in [n]$. If $l \geq b_{s} $ then $\mu_{B}(l)=s $. Since $a_{s} \leq b_{s} \leq l$, we also have $\mu_{A}(l) \geq s =\mu_{B}(l)$. If $l < b_{1}$ then $\mu_{B}(l) =0$, which proves the required result. Finally, consider $b_{1} \leq l < b_{s}$. Assume $|B \cap [l]|=p$. This means that there are p numbers $b_{1},\ldots,b_{p}$ that do not exceed $l$. Since $a_{1} \leq b_{1},\ldots,a_{p} \leq b_{p} \leq l$, we conclude that there are at least $p$ numbers $a_{1},\ldots,a_{p}$ that do not exceed $l$. Thus $\mu_{A}(l) \geq p =\mu_{B}(l)$.\\ 
(ii) Since  $A\leq B$, we conclude that $A \preceq B$ and from (i), we have $\mu_{A}(l) \geq \mu_{B}(l)$.\\
(iii) Since $A \subseteq B$, we have $B \preceq A$. From (i), we get $\mu_{B}(l) \geq \mu_{A}(l)$ for all $1 \leq l \leq n$.
\end{proof}
\begin{lem}\label{lem1} The family $\mathcal{F}(\mathcal{G})$ is intersecting if and only if $G \sim_{si} H $ for all $G,H \in \mathcal{G}$. 
\end{lem}
\begin{proof} Suppose $\mathcal{F}(\mathcal{G})$ is LCIF. We take any $G,H \in \mathcal{G}$. We need to prove that $G \sim_{si} H$. For $G' \leq G$ and $H' \leq H$, we set $A = G' \cup A'$ and $B=H' \cup B'$ with $A',B' \subseteq [n]\setminus (G' \cup H'), |A'|=k-|G'|, |B'|=k-|H'|$. Since $|A'|+|B'|=2k-(|G'|+|H'|) \leq n-(|G'|+|H'|$, it is possible to select $A'$ and $B'$
such that $A'\cap B'=\emptyset$. Now, we have $A \preceq G$ and $B \preceq H$. So $A \in \mathcal{F}(G)$ and $B \in \mathcal{F}(H)$. It follows that $A,B \in \mathcal{F}(\mathcal{G})$ and so $A \cap B \neq \emptyset$. From choosing $A',B'$ we have $G' \cap H' \neq \emptyset$. Thus, $G \sim_{si} H$. Conversely, suppose that $\mathcal{G}$ is strongly intersecting. We prove that $\mathcal{F}(\mathcal{G})$ is LCIF. Take $A,B \in \mathcal{F}(\mathcal{G})$. Then, there are $G,H \in \mathcal{G}$ such that $A \preceq G$ and $B \preceq H$. Assume that $|G|=r, |H|=s$ and $A=\{a_{1},\ldots,a_{r},\ldots,a_{k}\}, B=\{b_{1},\ldots,b_{s},\ldots,b_{k}\}$. We have $\{a_{1},\ldots,a_{r}\} \leq G$ and $\{b_{1},\ldots,b_{r}\} \leq H$. From $G \sim_{si} H$, it follows that $\{a_{1},\ldots,a_{r}\} \cap \{b_{1},\ldots,b_{r}\} \neq \emptyset$. Hence $A \cap B \neq \emptyset$.
\end{proof}
\begin{lem}\label{lem2} Let $G,H \in 2^{[n]}$. Then, the following statements are equivalent:\\
(i) $G$ and $H$ are strongly intersecting;\\
(ii) $\mathcal{F}(G)$ and $\mathcal{F}(H)$ are cross-intersecting;\\
(iii) There exists $1 \leq l \leq n$ such that $\mu_{G}(l)+\mu_{H}(l) > l$.
\end{lem}
\begin{proof} $(i) \Rightarrow (ii)$\\
Take $A \in \mathcal{F}(G), B \in \mathcal{F}(H)$. Then, $A \preceq G$ and $B \preceq H$. Let $A'$ be the set of the first $|G|$ elements of $A$ and $B'$ be the set of the first $|H|$ elements of $B$. Then, $A' \leq G, B' \leq H$. Since $G \sim_{si} H$, we get $A' \cap B' \neq \emptyset$ and therefore $A \cap B \neq \emptyset$.
Thus $\mathcal{F}(G)$ and $\mathcal{F}(H)$ are cross-intersecting.\\ \\
$(ii) \Rightarrow (iii)$\\
Suppose that $\mathcal{F}(G)$ and $\mathcal{F}(H)$ are cross-intersecting. For simplicity, for each $l \in [n]$, we denote $x_{l}=\mu_{G}(l), y_{l}=\mu_{H}(l), z_{l}=x_{l}+y_{l}$. We have the following simple observations. We have $x_{l+1}=x_{l}+1$ if $l+1 \in G$ and $x_{l+1}=x_{l}$ if $l+1 \notin G$. Similarly, we also have $y_{l+1}=y_{l}+1$ if $l+1 \in H$ and $y_{l+1}=y_{i}$ if $l+1 \notin H$. Therefore, $z_{l+1}-z_{l} \in \{0,1,2\}$.\\
Assume, for a contradiction, that $\mu_{G}(l)+\mu_{H}(l) \leq l$ for all $1 \leq l \leq n$. Thus, $z_{l} \leq l$, for all $1 \leq l \leq n$. By induction, for each $l$, we will construct the sets $G_{l}$ and $H_{l}$ satisfying simultaneously the following conditions: $G_{l} \leq G \cap [l], H_{l} \leq H \cap [l], G_{l} \cap H_{l}= \emptyset$ and $G_{l} \cup H_{l}=[z_{l}]$.\\ 

Consider $l=1$. We have $z_{1} \leq 1$. So $z_{1}=0$ or $z_{1}=1$. If $z_{1}=0$ then we take $G_{1}=H_{1}= \emptyset$ and we see that the above conditions are trivially satisfied. We consider the case where $z_{1}=x_{1}+y_{1}=1$. We may assume that $x_{1}=1$ and $y_{1}=0$. This means that $1 \in G$ and $1 \notin H$. We take $G_{1}=\{1\} \leq G \cap [1]$ and $H_{1}= \emptyset \leq H \cap [1]$. The above conditions are also satisfied.\\
Assume that we built two sets $G_{l}$ and $H_{l}$ that satisfy the conditions $G_{l} \leq G \cap [l], H_{l} \leq H \cap [l], G_{l} \cap H_{l}= \emptyset$ and $G_{l}\cup H_{l}=[z_{l}]$. We will build the sets $G_{l+1}$ and $H_{l+1}$. There are three possible cases. \\
\textbf{ The case of $z_{l+1}=z_{l}$}\\
We have $l+1 \notin G\cup H$. We set $G_{l+1}=G_{l}$ and $H_{l+1}=H_{l}$. We have, by the induction hypothesis, $G_{l+1}=G_{l} \leq G \cap [l]= G \cap [l+1]$ and $H_{l+1}=H_{l} \leq H \cap [l]= H \cap [l+1]$. Furthermore, $G_{l+1} \cap H_{l+1}= G_{l} \cap H_{l}= \emptyset$ and $G_{l+1} \cup H_{l+1}=G_{l} \cup H_{l}=[z_{l}]=[z_{l+1}]$. \\
\textbf{ The case of $z_{l+1}=z_{l}+1$}\\
We have $x_{l+1}=x_{l}+1, y_{l+1}=y_{l}$ or $y_{l+1}=y_{l}+1, x_{l+1}=x_{l}$. Without losing generality, we can see that $x_{l+1}=x_{l}+1, y_{l+1}=y_{l}$. Thus, $l+1 \in G$ and $l+1 \notin H$. We set $G_{l+1}=G_{l} \cup \{z_{l+1}\}$ and $H_{l+1}=H_{l}$. According to the induction hypothesis, $G_{l} \leq G \cap [l]$ and $H_{l} \leq H \cap [l]$. Note that $z_{l+1} \leq l+1$. It implies $G_{l+1}=G_{l} \cup \{z_{l+1}\} \leq G\cap [l+1]$ and $H_{l+1}=H_{l} \leq H\cap [l]=H \cap [l+1]$.   Also, by the induction hypothesis, we have $G_{l+1} \cap H _{l+1}= (G_{l} \cup \{z_{l+1}\}) \cap H_{l}=(G_{l} \cap H_{l})\cup (H_{l} \cap \{z_{l+1}\}=\emptyset$ since $z_{l+1} \notin H_{l}$. We also have $G_{l+1} \cup H_{l+1}=(G_{l} \cup H_{l}) \cup \{z_{l+1}\}=[z_{l}] \cup \{z_{l+1}\}=[z_{l+1}]$.\\ 
\textbf{ The case of $z_{l+1}=z_{l}+2$}\\
In this case, we have $x_{l+1}=x_{l}+1$ and $ y_{l+1}=y_{l}+1$. So $l+1 \in G \cap H$. We set $G_{l+1} =G_{l} \cup \{z_{l}+1\}$ and $H_{l+1} =H_{l} \cup \{z_{l}+2\}$. We check all the conditions. First, $G_{l+1} \cup H_{l+1} = (G_{l} \cup \{z_{l+1}-1\}) \cup (H_{l} \cup \{z_{l+1}\})= (G_{l} \cup H_{l}) \cup \{z_{l+1}-1\}  \cup \{z_{l+1}\}=[z_{l}]\cup \{z_{i+1}-1\} \cup\{z_{l+1}\}=[z_{l+1}]$. Next, since $G_{i} \cap H_{l}= \emptyset$, we have $G_{l+1} \cap H_{l+1}= (G_{l} \cup \{z_{l}+1\}) \cap (H_{l} \cup \{z_{l}+2\})= \emptyset$. Finally, since $l+1 \in G \cap H$ and $G_{l} \leq G \cap [l]$, we have $G_{l+1}=G_{l} \cup \{z_{l}+1\} \leq (G \cap [l]) \cup \{l+1\}=G \cap [l+1]$. Similarly, we also have $H_{l+1} \leq H \cap [l+1]$.\\ 
Now, we let $m = \text{ max } \{g_{1},\ldots,g_{r}, h_{1},\ldots,h_{s}\}$. Note that $G \cap [m] =G, H \cap [m]=H$ and $m \leq |G|+|H|$. Applying the above construction for $l=m$, we get two sets $G_{m}, H_{m}$ satisfying the conditions: $G_{m} \leq G, H_{m} \leq H, G_{m} \cup H_{m}=[m]$ and $G_{m} \cap H_{m}=\emptyset$. We take two sets $K,J \subseteq[m+1,n]$ such that $|K|=k-|G|, |J|=k-|H|$ and $K \cap J=\emptyset$ (it is possible to choose two sets $K$ and $J$ because $k-|G|+k-|H| \leq n-m )$. We define $A=G_{m} \cup K$ and $B=H_{m} \cup J$. We have $|A|=|B|=k$, $A \preceq G$ and $B \preceq H$. Thus, $A \in \mathcal{F}(G)$ and $B \in \mathcal{F}(G)$. Since $\mathcal{F}(G)$ and $\mathcal{F}(H)$ are cross-intersecting, we have $A \cap B \neq \emptyset$.
However, this leads to a contradiction, since by the construction of the sets $A$ and $B$, we clearly have $A \cap B =\emptyset$. Thus, there exists $l \in [n]$ such that $\mu_{G}(l)+\mu_{H}(l) > l$ as desired.\\
$(iii) \Rightarrow (i)$\\
Assume that there exists $l \in [n]$ such that $\mu_{G}(l) +\mu_{H}(l) > l$. We need to prove that $G$ and $H$ are strongly intersecting. Take any $G' \leq G$ and $H' \leq H$. By Proposition 2.1, we have $\mu_{G'}(l) \geq \mu_{G}(l)$ and $\mu_{H'}(l) \geq \mu_{H}(l)$. Thus, $\mu_{G'}(l) +\mu_{H'}(l) \geq \mu_{G}(l) +\mu_{H}(l) > l$. This implies that $(G' \cap [l])\cap (H' \cap [l]) \neq \emptyset$ and therefore, $G' \cap H' \neq \emptyset$. Thus, $G \sim_{si} H$. 
\end{proof}
For $G,H \in 2^{[n]}$, if there exists a number $1 \leq l \leq n$ such that 
$\mu_{G}(l)+\mu_{H}(l) > l$, then we say that {\bf \emph{ $G$ and $H$ are strongly intersecting at $l$ }} and write $G \sim_{si} H$ at $l$.\\ \\
\textbf{ Proof of Theorem 1.1}\\
From the Lemma \ref{lem1}, we see that the family $\mathcal{F}(\mathcal{G})$ is intersecting if and only if $\mathcal{G}$ is strongly intersecting. This means that, $G \sim_{si} H$ for every $G,H \in \mathcal{G}$. From the Lemma \ref{lem2}, we have that $G \sim_{si} H$ if anđ only if there exists a number $1 \leq l \leq n$ such that 
$\mu_{G}(l)+\mu_{H}(l) > l$. Thus, Theorem \ref{thr} follows directly from the lemmas \ref{lem1} and \ref{lem2}.
\subsection{Comparisons}
{There are some comparisons between Theorem 1.1 and The Condition of Bond. Firstly, Theorem 1.1 provides a more 'flexible' condition, as it only applies to subsets of 
$[n]$ (specifically, subsets of 
$[2k]$), whereas The Condition of Bond imposes a condition on the 
$k$-subsets of the left-compressed family itself. Thus, Bond’s condition follows as a consequence of Theorem 1.1. We prove this as follows. Let $\mathcal{A}$ be a left-compressed family. It is obvious that we only need to prove that if $\mathcal{A}$ is intersecting then for any $A=\{a_{1},\ldots,a_{k}\} \in \mathcal{A}$ and $B=\{b_{1},\ldots,b_{k}\} \in \mathcal{A}$, there exists $1 \leq i,j \leq k$ such that $i+j > max \{a_{i},b_{j}\}$. Since $\mathcal{A}$ is left-compressed, $\mathcal{A}$ can be described in terms of its maximal elements. Let $\mathcal{G}$ be the collection of all maximal of $\mathcal{A}$. Note that each $G \in \mathcal{G}$ is a $k$-set. We can see that $\mathcal{A}$ is generated by $\mathcal{G}$. Indeed, for any $A \in \mathcal{A}$, there exists $G \in \mathcal{G}$ such that $A \leq G$. It implies $A \preceq G$ and so $A \in  \mathcal{F}(\mathcal{G})$. Conversely, take any $A \in \mathcal{F}(\mathcal{G})$. We have $A \preceq G$ for some $G \in \mathcal{G}$. Since both $A$ and $G$ are $k$-set, it follows that $A \leq G$. Since $\mathcal{A}$ is left-compressed and $G \in \mathcal{A}$, we have $A \in \mathcal{A}$. Hence, $\mathcal{A}=\mathcal{F}(\mathcal{G})$. Now, if $\mathcal{A}$ is intersecting, then by Theorem 1.1, $\mathcal{G}$ is strongly intersecting. Take any $A,B \in \mathcal{A}$, then $A \leq G$ and $B \leq H$ for some $G,H \in \mathcal{G}$. We see that there exists $l$ such that $\mu_{G}(l)+\mu_{H}(l) >l$. From Proposition 2.1, we have $\mu_{A}(l) \geq \mu_{G}(l)$ and $\mu_{B}(l) \geq \mu_{H}(l)$. Then $\mu_{A}(l)+\mu_{B}(l) > l$. Let $i=\mu_{A}(l)$ and $j=\mu_{B}(l)$. We have $a_{i} \leq l$ and $b_{j} \leq l$ and it is clear that $i+j > l \geq max\{a_{i},b_{j}\}$.\\
Secondly, Theorem 1.1 supports a more useful tool to construct a left-compressed intersecting family. We only need to check the conditions of sets in $\mathcal{G}$. Note that each set in $\mathcal{G}$ is not necessarily a $k$-subset, so the construction of $\mathcal{G}$ is not complicated. Finally, the result of Bond is based on a property of Frankl ( see \cite{Fr87}). Our proof of Theorem 1.1 is based on the construction of two sets in two separate cross-intersecting families.}\\

\subsection{ Proof of Theorem 1.2}
 In the following, we present the proof of Theorem 1.2, which shows that an LCIF can be extended through suitable modifications of its generators.
\begin{proof} From Theorem 1.1, since $\mathcal{F}(\mathcal{G})$ is an LCIF, we have that $\mathcal{G}$ is strongly intersecting. Take any $A \in \mathcal{F}(\mathcal{G})$, this implies that $A \preceq G$ for some $G \in \mathcal{G}$. We know that $G \preceq \pi(G)$ and so $A \preceq \pi(G)$. It follows that $A \in \mathcal{F}(\pi(\mathcal{G}))$. Next, we prove that $\mathcal{\pi(\mathcal{G}})$ is strongly intersecting. We take any $G,H \in \pi(\mathcal{G})$. We have $G=\pi(G')$ and $H=\pi(H')$ for some $G', H' \in \mathcal{G}$. Since $\mathcal{G}$ is strongly intersecting, from Theorem 1.1 we have $\mu_{G'}(l)+\mu_{H'}(l) > l$, for some $1 \leq l \leq n$. Assume that $G'=\{g_{1},\ldots,g_{r},\ldots,g_{|G'|}\}$ has type $r \leq |G'|$. We prove that $l \leq k+r-1$. First, we consider $r=|G'|$. If $l \geq k+r$ then $\mu_{G'}(l) \leq r \leq l-k$. We have $\mu_{G'}(l)+\mu_{H'}(l) \leq l-k+k=l$, which is a contradiction. Next, we consider $r <|G'|$. We have $g_{r} < k+r <k+r+1 \leq g_{r+1}$. Therefore $g_{t} \geq k+t$ for all $r+1 \leq t \leq |G'|$. We will consider all possible cases for $l$. If $l=k+r$ then $\mu_{G'}(l)=l-k$ and thus, we again reach a contradiction. If $l \geq g_{|G'|}$ then $\mu_{G'}(l)=|G'| \leq g_{|G'|}-k$. It follows that $\mu_{G'}(l)+\mu_{H'}(l) \leq g_{|G'|} -k +k = g_{|G'|} \leq l$. Finally, we consider $g_{r+1} \leq l < g_{|G'|}$. We find that there exists an index $r+1 \leq t < |G'|$ such that $g_{t} \leq l < g_{t+1}$. We have $\mu_{G'}(l) =t \leq g_{t}-k$. So, $\mu_{G'}(l)+\mu_{H'}(l) \leq g_{t}-k+k=g_{t} \leq l$. This is a contradiction. Hence, in all cases, we have $l \leq k+r-1$. Now, since $G=\pi(G')=\{g_{1},\ldots,g_{r}\}$ and $g_{r} \leq k+r-1$, we have $\mu_{G}(l)=\mu_{G'}(l)$. Similarly, we also have $\mu_{H}(l)=\mu_{H'}(l)$. Thus, $\mu_{G}(l)+\mu_{H}(l)=\mu_{G'}(l)+\mu_{H'}(l) >l$. Consequently, from Theorem 1.1, we have $G \sim_{si} H$ at $l$ and therefore $\pi(\mathcal{G})$ is strongly intersecting.
\end{proof}
    
\section{FURTHER DIRECTIONS}
{ It is known that if $\mathcal{G} \subseteq 2^{[n]},$ then $\mathcal{F}(\mathcal{G})$ is left-compressed. We also know that if $\mathcal{G}$ is strongly intersecting, then $\mathcal{F}(\mathcal{G})$ is LCIF. Lemma 3.1 states that if $\mathcal{F}(\mathcal{G})$ is MLCIF, then $|G\cap[k+|G|-1]|=|G|$ for all $G \in \mathcal{G}$. Is this condition sufficient for $\mathcal{F}(\mathcal{G})$ to be an MLCIF? Hence, it is natural to pose the following question.\\
\textbf{ Question 1.} \textit{ Under what conditions on $\mathcal{G}$ is the family $\mathcal{F}(\mathcal{G})$ an MLCIF?} \\ \\
Theorem 1.2 allows us to extend an LCIF by “enlarging” the families $\mathcal{F}(G)$, where each $G$ is a generator of it. Therefore, a natural question arises: by this method, can we obtain an MLCIF that contains the original family?\\
\textbf{ Question 2 .} \textit{ Describe how to extend an LCIF to an MLCIF.} \\

\end{document}